\documentclass{amsart}
\usepackage{amsthm,amssymb,amsmath}
\usepackage{hyperref}
\theoremstyle{definition}

\theoremstyle{plain}
\newtheorem{thm}{Theorem}
\newtheorem{lem}{Lemma}
\begin{document}
\title{A note on five dimensional kissing arrangements}
\author{Ferenc Sz\"oll\H{o}si}
\thanks{Date: \today. Preprint}
\address{Department of Mathematical Sciences, Interdisciplinary Faculty of Science and Engineering, Shimane University, Matsue, Shimane, 690-8504, Japan}
\email{szollosi@riko.shimane-u.ac.jp}
\begin{abstract}
The kissing number $\tau(d)$ is the maximum number of pairwise non-overlapping unit spheres each touching a central unit sphere in the $d$-dimensional Euclidean space. In this note we report on how we discovered a new, previously unknown arrangement of $40$ unit spheres in dimension $5$. Our arrangement saturates the best known lower bound on $\tau(5)$, and refutes a `belief' of Cohn--Jiao--Kumar--Torquato.
\end{abstract}
\maketitle
\section{Introduction}
This paper is based on a talk given by the author during the ``Workshop on Orthogonal designs and related Combinatorics'' held in Meiji University in Tokyo on January 21, 2023.

Given a central unit sphere in the $d$-dimensional Euclidean space $\mathbb{R}^d$, surrounded by $n\geq 1$ pairwise non-overlapping unit spheres at distance $2$ is called a \emph{kissing arrangement}. The main problem of interest is the determination of the maximum number $\tau(d)$ of such spheres. This is a challenging open problem, called the \emph{kissing number problem} in geometry. For background and references we refer the reader to \cite{CJKT}, \cite{CS}, \cite{MIK}, and \cite{M}.

It is useful to reformulate this problem in the language of algebra. The kissing number problem asks for the maximum number $n$ of unit vectors $v_1$, $v_2$, $\dots$, $v_n$ in $\mathbb{R}^d$, such that their pairwise inner product satisfies $\left\langle v_i,v_j\right\rangle\leq \frac{1}{2}$ for every $i\neq j\in\{1,2,\dots,n\}$. It is easy to see that the vertices of the cross-polytope, formed by the union of the standard basis and its negative, is a kissing arrangement, thus $\tau(d)\geq 2d$. The exact value is only known for $d\in\{1,2,3,4,8,24\}$.

In this paper we report on some computational techniques leading to the discovery of a new, previously unknown kissing arrangement of $40$ spheres in $\mathbb{R}^5$. To this date this is the third such known arrangement in dimension $5$ (see reference \cite{L} from 1967) and testifies our extremely limited understanding of the structure of optimal high-dimensional sphere packings.

\section{The best known kissing arrangements in dimension five}
Since the value of $\tau(d)$ has been determined up to $d\leq 4$, see \cite{M}, the first open case is dimension $5$. Here two kissing arrangements showing $\tau(5)\geq 40$ were previously known \cite{CJKT}, \cite{KZ}, \cite{L}. We briefly describe these arrangements as follows.

First, the set of all $2d(d-1)$ unit vectors in $\mathbb{R}^d$, $d\geq 2$ having exactly two non-zero coordinates of equal absolute value $1/\sqrt2$
\[\mathcal{D}_d:=\{\sigma([\pm1/\sqrt2,\pm1/\sqrt2,0,0,\dots,0])\colon\sigma\in S_d\}\]
forms an antipodal kissing arrangement: if $v\in\mathcal{D}_d$ then $-v\in\mathcal{D}_d$. Fixing $d=5$ shows that $\tau(5)\geq 40$. The profile (i.e., the multiset of pairwise inner products) of $\mathcal{D}_5$ is given by:
\[\pi(\mathcal{D}_5)=\{[-1]^{40},[-1/2]^{480},[0]^{560},[1/2]^{480},[1]^{40}\}.\]
The profile is clearly an invariant up to change of basis.

Rather interestingly, another example can be obtained by rotating the $8$ vectors among $\mathcal{D}_5$ with last coordinate $-1/\sqrt2$ by the orthogonal matrix
\[H=\frac{1}{2}\left[\begin{array}{rrrrr}
-1 & 1 & 1 & 1 & 0\\
1 & -1 & 1 & 1 & 0\\
1 & 1 & -1 & 1 & 0\\
1 & 1 & 1 & -1 & 0\\
0 & 0 & 0 & 0 & 2\end{array}\right].\]

This gives rise to the non-antipodal kissing arrangement (see \cite{L})
\[\mathcal{L}_5:=(\mathcal{D}_5\setminus\{v\in\mathcal{D}_5\colon v_5=-1/\sqrt2\})\cup\{vH\colon v\in\mathcal{D}_5, v_5=-1/\sqrt2\}\]

with profile
\[\pi(\mathcal{L}_5)=\{[-1]^{24},[-3/4]^{64},[-1/2]^{384},[-1/4]^{64},[0]^{544},[1/2]^{480},[1]^{40}\}.\]
It was believed that these two arrangements are the only possibilities \cite{CJKT}, \cite{CS}. In the next section we are going to mention three computer-aided methods to discover kissing arrangements, and use them to construct a new, previously unknown arrangement of $40$ unit spheres in dimension $5$.

\section{In search of kissing arrangements}
We call the distinct unit vectors $x,y\in\mathbb{R}^d$ \emph{compatible}, if they represent non-overlapping spheres, that is if $\left\langle x,y\right\rangle\leq \frac{1}{2}$.
\subsection{Method 1: The multiangular cloud}\label{ss1} Fix a dimension $d$, a (finite!) set of inner products $\mathcal{A}$ (with $\max\mathcal{A}\leq\frac{1}{2})$, and a $d$-dimensional basis $B_1$, $B_2$, $\dots$, $B_d$ of pairwise compatible unit vectors, represented by the rows of a $d\times k$ matrix $B$. The Gram matrix of the basis vectors $G=BB^T$ is an invertible $d\times d$ matrix.

Now assume that a unit vector $v\in\mathbb{R}^k$ is (i) in the span of the basis vectors; and (ii) not only compatible with each of the basis vectors, but rather $\left\langle v, B_i\right\rangle \in \mathcal{A}$ for every $i\in\{1,\dots,d\}$. Then, the first condition says that we have a coordinate vector $c\in\mathbb{R}^d$, such that $v=cB$. The second condition says that we have a $d$-tuple $t\in\mathcal{A}^d$, such that $Bv^T=t^T$. Combining these two we obtain $v=tG^{-1}B$. Since $v$ must be a unit vector, we have in addition $tG^{-1}t^{T}=1$. We call the (finite) set of all such vectors
\[\mathcal{C}_{\mathcal{A},B}:=\{tG^{-1}B\colon t\in\mathcal{A}^d, tG^{-1}t^{T}=1\}\subset\mathbb{R}^{k}\]
the \emph{multiangular cloud surrounding the basis $B$ with respect to the angle set $\mathcal{A}$}. Clearly, any subset of pairwise compatible vectors of $\mathcal{C}_{\mathcal{A},B}$ together with the rows of $B$ forms a kissing arrangement.

We are interested in finding the maximum (that is: the largest) subset(s). This in turn boils down to a clique search \cite{PO} in the graph $\Gamma$, where the vertex set is represented by the vectors in the cloud, and edges are precisely between compatible vectors. The size of the maximum clique in $\Gamma$ is denoted by $\omega(\Gamma)$ as usual.
\begin{lem}
With the notations introduced in Section~\ref{ss1}, we have $\tau(d)\geq d+\omega(\Gamma)$.
\end{lem}
\subsection{Method 2: Extending a basis of $\mathbb{R}^d$ to a basis of $\mathbb{R}^{d+1}$}\label{ss2}
Let $d$, $t$, $\mathcal{A}$, $B$, and $G$ be the same as defined in Section~\ref{ss1}. Assume further that $k=d+1$, and therefore there is a unit vector $B_{d+1}\in\mathbb{R}^{d+1}$ orthogonal to each vectors of the linearly independent set $B$. Once again we consider vectors of the form $v=tG^{-1}B\in\mathbb{R}^{d+1}$ in the linear span of the vectors $B$. However, if this $v$ is not a unit a vector, but shorter, then we may consider instead the two vectors $v\pm\sqrt{1-\left\langle v,v\right\rangle}B_{d+1}$. The finite set of all such vectors
\[\mathcal{C}_{\mathcal{A},B}':=\{tG^{-1}B\pm\sqrt{1-tG^{-1}t^T}B_{d+1}\colon t\in\mathcal{A}^d, tG^{-1}t^{T}\leq 1\}\subset\mathbb{R}^{d+1},\]
is the multiangular cloud in $\mathbb{R}^{d+1}$ surrounding the subset $B$ with respect to the angle set $\mathcal{A}$. Similarly, a maximum subset of $\mathcal{C}_{\mathcal{A},B}'$ of pairwise compatible vectors can be found using a clique search in the corresponding graph $\Gamma'$.
\begin{lem}\label{lem2}
With the notations introduced in Section~\ref{ss2}, we have $\tau(d+1)\geq d+\omega(\Gamma')$.
\end{lem}

\subsection{Method 3: Precomputed vectors}\label{ss3}
Let $d$ be fixed, and let $\mathcal{C}''\subset\mathbb{R}^d$ be a finite set of unit vectors. Let $\Gamma''$ be a graph on the vertex set of $\mathcal{C}''$, where two vertices are connected if and only if the vectors are compatible. A maximum clique in $\Gamma''$ yields a kissing arrangement.
\begin{lem}
With the notations introduced in Section~\ref{ss3}, we have $\tau(d)\geq \omega(\Gamma'')$.
\end{lem}
\section{A new kissing arrangement in dimension five}
We illustrate our Method 2 as follows. Let $d=4$, let $\mathcal{A}=\{-1,-1/2,0,1/2\}$, and let
\[B=\frac{1}{\sqrt{50}}\left[\begin{array}{rrrrr}
5 & 5 & 0 & 0 & 0\\
5 & 0 & 5 & 0 & 0\\
5 & 0 & 0 & 5 & 0\\
5 & 0 & 0 & 0 & 5
\end{array}\right],\quad G=BB^T=\frac{1}{2}\left[
\begin{array}{cccc}
2 & 1 & 1 & 1\\
1 & 2 & 1 & 1\\
1 & 1 & 2 & 1\\
1 & 1 & 1 & 2
\end{array}
\right],\]
and $B_5=\frac{1}{\sqrt5}[-1,1,1,1,1]$. The cloud $\mathcal{C}_{\mathcal{A},B}'\subset\mathbb{R}^5$ contains $78$ vectors. A routine application of the cliquer C software \cite{PO} shows that the graph $\Gamma'$ contains exactly four maximum cliques of size $\omega(\Gamma')=36$. Two of these cliques yield via Lemma~\ref{lem2} a set of $40$ vectors with the same profile as of $\mathcal{D}_5$. The other two, however, yields $40$ vectors with a distinct profile. Denoting one of these two by $\mathcal{Q}_5$, we obtain:
\[\pi(\mathcal{Q}_5)=\{[-1]^{20},[-4/5]^{60},[-1/2]^{360},[-3/10]^{120},[0]^{500},[1/5]^{20},[1/2]^{480},[1]^{40}\}.\]
This implies that the arrangement $\mathcal{Q}_5$ is not isometric to $\mathcal{D}_5$ or $\mathcal{L}_5$. We record this as our main result.
\begin{thm}
In dimension $5$ there exist at least three, pairwise non-isometric kissing-arrangements of $40$ vectors: $\mathcal{D}_5$, $\mathcal{L}_5$, and  $\mathcal{Q}_5:=(\mathcal{D}_5\setminus\mathcal{X})\cup\mathcal{Y}$, where the sets $\mathcal{X}$ and $\mathcal{Y}$ are the row vectors of the matrices
\[X=\frac{1}{\sqrt2}\left[\begin{array}{rrrrr}
1 & -1 & 0 & 0 & 0\\
1 & 0 & -1 & 0 & 0\\
1 & 0 & 0 & -1 & 0\\
1 & 0 & 0 & 0 & -1\\
0 & -1 & -1 & 0 & 0\\
0 & -1 & 0 & -1 & 0\\
0 & -1 & 0 & 0 & -1\\
0 & 0 & -1 & -1 & 0\\
0 & 0 & -1 & 0 & -1\\
0 & 0 & 0 & -1 & -1\end{array}\right],\qquad Y=\frac{1}{5\sqrt2}\left[\begin{array}{rrrrr}
-1 & 1 & -4 & -4 & -4\\
-1 & -4 & 1 & -4 & -4\\
-1 & -4 & -4 & 1 & -4\\
-1 & -4 & -4 & -4 & 1\\
4 & 1 & 1 & -4 & -4\\
4 & 1 & -4 & 1 & -4\\
4 & 1 & -4 & -4 & 1\\
4 & -4 & 1 & 1 & -4\\
4 & -4 & 1 & -4 & 1\\
4 & -4 & -4 & 1 & 1
\end{array}\right].\]
\end{thm}
\begin{proof}
One readily checks that $\mathcal{Q}_5$ is a set of $40$ pairwise compatible unit vectors in dimension $5$. Further, the profile of the listed configurations are mutually different.
\end{proof}
We remark that $\mathcal{Q}_5$ contains only $20$ antipodal vectors, therefore it does not contain an isometric copy of the $24$ antipodal vectors of $\mathcal{D}_4$.

It would be interesting to analyze the rigidity \cite{CJKT} of $\mathcal{Q}_5$, and explore higher dimensional kissing arrangements containing an isometric copy of it as a subset.

\section*{Appendix: The vectors forming $\mathcal{Q}_5$}
We list the vectors of $Q_5$ in machine-readable form. The vectors below should be rescaled by a factor of $\frac{1}{5\sqrt{2}}$.
\begin{verbatim}
Q5={{5,5,0,0,0},{5,0,5,0,0},{5,0,0,5,0},{5,0,0,0,5},{0,0,0,-5,5},
{0,0,0,5,-5},{0,0,0,5,5},{0,0,-5,0,5},{0,0,-5,5,0},{0,0,5,0,-5},
{0,0,5,0,5},{0,0,5,-5,0},{0,0,5,5,0},{0,-5,0,0,5},{0,-5,0,5,0},
{0,-5,5,0,0},{0,5,0,0,-5},{0,5,0,0,5},{0,5,0,-5,0},{0,5,0,5,0},
{0,5,-5,0,0},{0,5,5,0,0},{-5,0,0,0,-5},{-5,0,0,0,5},{-5,0,0,-5,0},
{-5,0,0,5,0},{-5,0,-5,0,0},{-5,0,5,0,0},{-5,-5,0,0,0},{-5,5,0,0,0},
{-1,1,-4,-4,-4},{-1,-4,1,-4,-4},{-1,-4,-4,1,-4},{-1,-4,-4,-4,1},
{4,1,1,-4,-4},{4,1,-4,1,-4},{4,1,-4,-4,1},{4,-4,1,1,-4},
{4,-4,1,-4,1},{4,-4,-4,1,1}};
\end{verbatim}
\end{document}